\documentclass[10pt,a4paper]{article}
\usepackage[margin=1.1cm]{geometry}
\usepackage{amsmath,amsthm,amsfonts,amssymb,amscd,cite,graphicx}
\usepackage{tikz,xcolor}
\usetikzlibrary{patterns}
\usepackage{graphics}
\usepackage{caption, subcaption}
\usepackage{enumitem}
\usepackage{algorithm}
\usepackage{algpseudocode}
\usepackage{hyperref}

\usepackage{titlesec}
\titleformat{\section}
{\normalfont\fontsize{12}{15}\bfseries}{\thesection}{1em.}{}

\newtheorem{proposition}{Proposition}[section]

\newtheorem{lemma}{Lemma}[section]
\newtheorem{definition}{Definition}[section]
\newtheorem{theorem}{Theorem}[section]
\newtheorem{remark}{Remark}[section]
\newtheorem{example}{Example}[section]

\newtheorem{fact}{Fact}[section]

\allowdisplaybreaks[4]

\let\oldbibliography\thebibliography
\renewcommand{\thebibliography}[1]{%
  \oldbibliography{#1}%
  \setlength{\itemsep}{-2pt}%
}

\newcommand{\tdot}[3]{\draw [fill=black,color=#3] (#1,#2) circle [radius=0.25];}

\newcommand{\beq}{\begin{equation}}
\newcommand{\eeq}{\end{equation}}

\newcommand{\beqn}{\begin{eqnarray}}
\newcommand{\eeqn}{\end{eqnarray}}

\newcommand{\luk}{Łukasiewicz}
\newcommand{\PF}[1]{\mathrm{PF}_{#1}}
\newcommand{\incPF}[1]{\mathrm{PF}^{\mathrm{inc}}_{#1}}
\newcommand{\PPF}[1]{\mathrm{PrimePF}_{#1}}
\newcommand{\incPPF}[1]{\mathrm{PrimePF}^{\mathrm{inc}}_{#1}}
\newcommand{\UPF}[1]{\mathrm{UPF}_{#1}}
\newcommand{\incUPF}[1]{\mathrm{UPF}^{\mathrm{inc}}_{#1}}
\newcommand{\Luk}[1]{\mathrm{Luk}_{#1}}
\newcommand{\PLuk}[1]{\mathrm{PrimeLuk}_{#1}}
\newcommand{\disp}[1]{\mathrm{disp}\left(#1\right)}
\newcommand{\area}[1]{\mathrm{Area}\left(#1\right)}
\newcommand{\height}[1]{\mathrm{Height}\left(#1\right)}
\newcommand{\PFtoLuk}{\Psi_{\mathrm{PF} \rightarrow \mathrm{Luk}}}
\newcommand{\LuktoPF}{\Psi_{\mathrm{Luk} \rightarrow \mathrm{PF}}}

\newcommand{\OO}[1]{\mathcal{O}\left( #1 \right)}
\newcommand{\Cat}[1]{\mathrm{Cat}_{#1}}
\newcommand{\Motz}[1]{\mathrm{Motz}_{#1}}

\definecolor{mygreen}{rgb}{0.1,0.8,0.1}
\definecolor{mygray}{rgb}{0.7,0.7,0.7}

\begin{document}

\baselineskip=0.20in

\makebox[\textwidth]{%
\hglue-15pt
\begin{minipage}{0.6cm}	
\vskip9pt
\end{minipage} \vspace{-\parskip}
\begin{minipage}[t]{6cm}
\footnotesize{ {\bf Discrete Mathematics Letters} \\ \underline{www.dmlett.com}}
\end{minipage}
\hfill
\begin{minipage}[t]{6.5cm}
\normalsize {\it Discrete Math. Lett.}  {\bf 14} (2024) 77--84
\end{minipage}}
\vskip36pt

\noindent
{\large \bf Parking functions and \luk\ paths}\\

\noindent
Thomas Selig$^{1,}\footnote{Corresponding author (Thomas.Selig@xjtlu.edu.cn).}$, Haoyue Zhu$^{1}$\\

\noindent
\footnotesize $^1${\it Department of Computing, School of Advanced Technology, Xi'an Jiaotong-Liverpool University, Suzhou, China}

\noindent
 (\footnotesize Received: 25 March 2024. Received in revised form: 29 October 2024. Accepted: 4 November 2024. Published online: 6 November 2024.)\\

\setcounter{page}{1} \thispagestyle{empty}

\baselineskip=0.20in

\normalsize

 \begin{abstract}
 \noindent
 We present a bijection between two well-known objects in the ubiquitous Catalan family: non-decreasing parking functions and \luk\ paths. This bijection maps the maximum displacement of a parking function to the height of the corresponding \luk\ path, and the total displacement to the area of the path. We also study this bijection restricted to two specific families of parking functions: unit-interval parking functions and prime parking functions.
 \\[2mm]
 \end{abstract}

\baselineskip=0.20in

\section{Introduction}\label{sec:intro}

The Catalan numbers, defined by $\Cat{n} := \frac{1}{n+1} \binom{2n}{n}$, are ubiquitous in combinatorics. The corresponding entry, Sequence A000108 in the OEIS~\cite{OEIS}, states: ``This is probably the longest entry in the OEIS, and rightly so''.  Stanley's book, \emph{Catalan Numbers}~\cite{StanCat}, offers hundreds of combinatorial interpretations of these numbers. Let us simply list some of the most famous: Dyck paths, plane trees, non-crossing partitions, permutations avoiding any single pattern of length $3$, and so on. 

In this paper, we present a bijection between two objects of the Catalan family: non-decreasing parking functions, and \luk\ paths (see Equation~\eqref{eq:PF_extendluk_path}). Surprisingly, while this bijection appears implicitly in some works (see e.g.~\cite[Section~1.4]{PFposet}), and can be obtained by composing various well-known bijections such as those in Stanley's book~\cite{StanCat}, it does not appear to have been explicitly stated in the literature. In particular, the preservation of various statistics (see below) does not seem to have been previously noted, and is one of our main contributions. 

Our paper is organised as follows. In the remainder of this section, we introduce parking functions and \luk\ paths, and recall some of their important properties. Section~\ref{sec:main} establishes the main result of this paper: a bijection between non-decreasing parking functions and \luk\ paths. We study the effects of the bijection on the \emph{displacement} of parking functions, which measures how far away cars end up from their preferred spots. More precisely, in Theorem~\ref{thm:bij_pf_luk} we will see that the total displacement of a parking function maps to the area of the corresponding \luk\ path, and the maximum displacement to the height. We also present a simple algorithmic procedure to get the inverse bijection from \luk\ paths to non-decreasing parking functions (Algorithm~\ref{algo:LuktoPF} and Theorem~\ref{thm:algo_LuktoPF}). Finally, in Section~\ref{sec:spec} we study various specialisations of this bijection to families of parking functions and \luk\ paths with added restrictions.

\subsection{Parking functions}\label{subsec:pf}

Throughout this paper, $n$ denotes a positive integer, and we let $[n] := \{1, \ldots, n\}$. Consider a one-directional car park consisting of $n$ spots labelled $1$ to $n$, and $n$ cars also labelled $1$ to $n$. A \emph{parking preference} is a vector $p=(p_1,\ldots,p_n) \in [n]^n$, with $p_i$ denoting the preferred parking spot of car $i$ for each $i \in [n]$. The cars enter the car park sequentially in order $1$ to $n$. If spot $p_i$ is empty when car $i$ enters, then car $i$ parks in spot $p_i$. Otherwise, if spot $p_i$ has already been occupied by some previous car $j<i$, car $i$ cannot park in spot $p_i$. In that case, the car drives on and parks in the first unoccupied spot $k>p_i$. If no such spot exists, car $i$ exits the car park and fails to park. We say that $p$ is a  \emph{parking function} if all cars are able to park through this process (see Example~\ref{exa:PF_outcome}). We denote by $\PF{n}$ the set of parking functions with $n$ cars/spots. 

\begin{definition}\label{def:spotSeq}
Given a parking function $p = (p_1, \ldots, p_n)$, the \emph{outcome} of $p$ is the sequence $\OO{p}=(o_1,\ldots,o_n) $, where for each car $i\in [n]$, $o_i$ is the spot where car $i$ ends up parking.
\end{definition}

\begin{definition}\label{def:displacement}
Given a parking function $p=(p_1,\ldots,p_n) \in \PF{n}$, with outcome $\OO{p} = (o_1, \ldots, o_n)$, the \emph{displacement} $d_i$ of car $i$ describes the distance between the initial preference $p_i$ of car $i$ and the actual spot where car $i$ ends up, i.e.\ $d_i=o_i-p_i$. The \emph{displacement vector} of $p$ is then $\disp{p} := (d_1,d_2,\ldots,d_n)$, and the \emph{total displacement} of $p$ is $\vert \disp{p} \vert := \sum_{i\in[n]}d_i$.
\end{definition}

\begin{example}\label{exa:PF_outcome}
Consider the parking preference $p = (2, 1, 4, 4, 1)$. We first describe the parking process for $p$, which is illustrated in Figure~\ref{fig:exa_OPF}. Initially car $1$ parks in spot $2$, followed by car $2$ parking in spot $1$, and car $3$ in spot $4$ (none of these spots are occupied when the cars arrive). When car $4$ arrives, it wants to park in spot $4$. However, spot $4$ is occupied by car $3$, so car $4$ drives on to find the first available spot and park in it, which is spot $5$. Finally, car $5$ wants to park in spot $1$, but spot $1$ has been occupied by car $2$, causing car $5$ to drive on: spot $2$ is also occupied by car $1$, so car $5$ ends up parking in spot $3$ (which is the first available spot at this point). Finally, all cars are able to park, so $p$ is a parking function. Moreover, we get the outcome $\OO{p}=(2,1,4,5,3)$, and displacement vector $\disp{p} = (0,0,0,1,2)$, with total displacement $\vert \disp{p} \vert = 3$.

\begin{figure}[ht]
 \centering
  \begin{tikzpicture}[scale=0.2]
    
    \node at (-3,1.2) {cars};
    \node at (-3,-4.5) {spots};       
    
    \foreach \x in {1,...,5}
	  \node at (-1+4*\x,-4.5) {$\x$};
    \draw [thick, color=blue] (1,-2)--(1,-3)--(5,-3)--(5,-2);
    \draw [thick, color=blue] (5,-2)--(5,-3)--(9,-3)--(9,-2);
    \draw [thick, color=blue] (9,-2)--(9,-3)--(13,-3)--(13,-2);
    \draw [thick, color=blue] (13,-2)--(13,-3)--(17,-3)--(17,-2);
    \draw [thick, color=blue] (17,-2)--(17,-3)--(21,-3)--(21,-2);
    \node [draw, circle, color=red, scale=0.8pt] (1) at (7,1.2) {$1$};
    \node at (7, -1.8) {$\downarrow$};
    \node at (25, 0) {$\Longrightarrow$};

    \begin{scope}[shift={(28,0)}]
    \foreach \x in {1,...,5}
	  \node at (-1+4*\x,-4.5) {$\x$};
    \draw [thick, color=blue] (1,-2)--(1,-3)--(5,-3)--(5,-2);
    \draw [thick, color=blue] (5,-2)--(5,-3)--(9,-3)--(9,-2);
    \draw [thick, color=blue] (9,-2)--(9,-3)--(13,-3)--(13,-2);
    \draw [thick, color=blue] (13,-2)--(13,-3)--(17,-3)--(17,-2);
    \draw [thick, color=blue] (17,-2)--(17,-3)--(21,-3)--(21,-2);
    \node [draw, circle, scale=0.8pt] (1) at (7,-1) {$1$};
    \node [draw, circle, color=red, scale=0.8pt] (1) at (3,1.2) {$2$};
    \node at (3, -1.8) {$\downarrow$};
    \node at (25, 0) {$\Longrightarrow$};
    \end{scope}

    \begin{scope}[shift={(56,0)}]
    \foreach \x in {1,...,5}
	  \node at (-1+4*\x,-4.5) {$\x$};
    \draw [thick, color=blue] (1,-2)--(1,-3)--(5,-3)--(5,-2);
    \draw [thick, color=blue] (5,-2)--(5,-3)--(9,-3)--(9,-2);
    \draw [thick, color=blue] (9,-2)--(9,-3)--(13,-3)--(13,-2);
    \draw [thick, color=blue] (13,-2)--(13,-3)--(17,-3)--(17,-2);
    \draw [thick, color=blue] (17,-2)--(17,-3)--(21,-3)--(21,-2);
    \node [draw, circle, scale=0.8pt] (1) at (7,-1) {$1$};
    \node [draw, circle, scale=0.8pt] (2) at (3,-1) {$2$};
    \node [draw, circle, color=red, scale=0.8pt] (3) at (15,1.2) {$3$};
    \node at (15, -1.8) {$\downarrow$};
    \end{scope}

    \begin{scope}[shift={(0,-12)}]
    \node at (-3, 0) {$\Longrightarrow$};
    \foreach \x in {1,...,5}
	  \node at (-1+4*\x,-4.5) {$\x$};
    \draw [thick, color=blue] (1,-2)--(1,-3)--(5,-3)--(5,-2);
    \draw [thick, color=blue] (5,-2)--(5,-3)--(9,-3)--(9,-2);
    \draw [thick, color=blue] (9,-2)--(9,-3)--(13,-3)--(13,-2);
    \draw [thick, color=blue] (13,-2)--(13,-3)--(17,-3)--(17,-2);
    \draw [thick, color=blue] (17,-2)--(17,-3)--(21,-3)--(21,-2);
    \node [draw, circle, scale=0.8pt] (1) at (7,-1) {$1$};
    \node [draw, circle, scale=0.8pt] (2) at (3,-1) {$2$};
    \node [draw, circle, scale=0.8pt] (3) at (15,-1) {$3$};
    \node [draw, circle, color=red, scale=0.8pt] (4) at (15,3) {$4$};
    \node at (18.2, 3) {$\curvearrowright$};
    \node at (25, 0) {$\Longrightarrow$};
    \end{scope}

    \begin{scope}[shift={(28,-12)}]
    \node at (-3, 0) {$\Longrightarrow$};
    \foreach \x in {1,...,5}
	  \node at (-1+4*\x,-4.5) {$\x$};
    \draw [thick, color=blue] (1,-2)--(1,-3)--(5,-3)--(5,-2);
    \draw [thick, color=blue] (5,-2)--(5,-3)--(9,-3)--(9,-2);
    \draw [thick, color=blue] (9,-2)--(9,-3)--(13,-3)--(13,-2);
    \draw [thick, color=blue] (13,-2)--(13,-3)--(17,-3)--(17,-2);
    \draw [thick, color=blue] (17,-2)--(17,-3)--(21,-3)--(21,-2);
    \node [draw, circle, scale=0.8pt] (1) at (7,-1) {$1$};
    \node [draw, circle, scale=0.8pt] (2) at (3,-1) {$2$};
    \node [draw, circle, scale=0.8pt] (3) at (15,-1) {$3$};
    \node [draw, circle, scale=0.8pt] (4) at (19,-1) {$4$};
    \node [draw, circle, color=red, scale=0.8pt] (5) at (3,3) {$5$};
    \node at (7, 3) {$\longrightarrow$};
    \node at (10.5, 2) {$\searrow$};
    \node at (25, 0) {$\Longrightarrow$};
    \end{scope}

    \begin{scope}[shift={(56,-12)}]
    \foreach \x in {1,...,5}
	  \node at (-1+4*\x,-4.5) {$\x$};
    \draw [thick, color=blue] (1,-2)--(1,-3)--(5,-3)--(5,-2);
    \draw [thick, color=blue] (5,-2)--(5,-3)--(9,-3)--(9,-2);
    \draw [thick, color=blue] (9,-2)--(9,-3)--(13,-3)--(13,-2);
    \draw [thick, color=blue] (13,-2)--(13,-3)--(17,-3)--(17,-2);
    \draw [thick, color=blue] (17,-2)--(17,-3)--(21,-3)--(21,-2);
    \node [draw, circle, scale=0.8pt] (1) at (7,-1) {$1$};
    \node [draw, circle, scale=0.8pt] (2) at (3,-1) {$2$};
    \node [draw, circle, scale=0.8pt] (3) at (15,-1) {$3$};
    \node [draw, circle, scale=0.8pt] (4) at (19,-1) {$4$};
    \node [draw, circle, scale=0.8pt] (5) at (11,-1) {$5$};
    \end{scope}
    
  \end{tikzpicture}
  \caption{The parking process for $p=(2, 1, 4, 4, 1) \in \PF{5}$.\label{fig:exa_OPF}}
\end{figure}
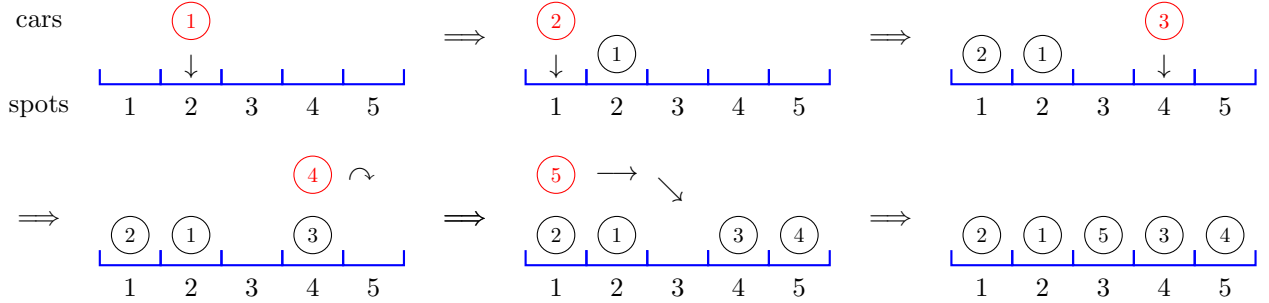

\end{example}

Parking functions were originally introduced by Konheim and Weiss~\cite{KonWeiss} to study a hashing technique known as linear probing. Since then, they have been a popular research topic in Mathematics and Computer Science, with rich combinatorial connections to fields such as hyperplane arrangements~\cite{StanHyp} or statistical physics~\cite{CR}. We refer the interested reader to the excellent survey by Yan~\cite{YanSurvey}. The following result gives two classical characterisations of parking functions (see~\cite{YanSurvey}).

\begin{proposition}\label{pro:pf_char}
Let $p = (p_1, \ldots, p_n) \in [n]^n$ be a parking preference. Define $p^{\mathrm{inc}} = (p^{\mathrm{inc}}_1, \ldots, p^{\mathrm{inc}}_n)$ be its non-decreasing re-arrangement. Then the following are equivalent.
\begin{enumerate}[noitemsep, topsep=2pt]
\item We have $p \in \PF{n}$.
\item For all $i \in [n]$, we have $p^{\mathrm{inc}}_i \leq i$.
\item For all $i \in [n]$, we have $\left\vert \{j \in [n]; \, p_j \leq i \} \right\vert \geq i$.
\end{enumerate}
\end{proposition}

In this paper, we will be primarily interested in \emph{non-decreasing} parking functions. These are parking functions which are in weakly increasing order, i.e.\ $p_i \leq p_{i+1}$ for all $i \in [n-1]$. We denote $\incPF{n}$ the set of non-decreasing parking functions of length $n$. There is a classical bijection between non-decreasing parking functions and Dyck paths, which we recall briefly here (see also~\cite[Page~54]{YanSurvey}). 

For our purposes, a Dyck path will be a lattice path from $(0,0)$ to $(n,n)$ for some $n \geq 0$ with steps $E = (1,0)$ and $N = (0,1)$, which never goes above the diagonal $y = x$ (see Figure~\ref{fig:PF_to_Dyck}). Note that a Dyck path $w$ is uniquely determined by the weakly increasing sequence $0 = h_1 \leq \cdots \leq h_n \leq n-1$ of heights ($y$-coordinates) of its $E$ steps. Formally, we define $w = w(h_1, \ldots, h_n) := E N^{h_2-h_1} E N^{h_3 - h_2} \cdots E N^{n-h_n}$, where the notation $N^k$ indicates the step $N$ repeated $k$ times.

It is then straightforward to see that the map $p = (p_1, \ldots, p_n) \mapsto w(p_1-1, \ldots, p_n-1)$ is a bijection from the set $\incPF{n}$ of non-decreasing parking functions of length $n$ to the set of Dyck paths with $2n$ steps. In particular, we have $\vert \incPF{n} \vert = \Cat{n}$, the $n$-th Catalan number. Moreover, for $p \in \incPF{n}$, the total displacement $\vert \disp{p} \vert$ of $p$ is equal to the \emph{area} $\area{w}$ of the corresponding Dyck path $w$, defined as the number of complete lattice squares between the path $w$ and the line $y=x$.

\begin{example}\label{ex:PF_to_Dyck}
Consider the non-decreasing parking function $p = (1,1,2,4,4) \in \incPF{5}$. The corresponding height sequence is $h = (0,0,1,3,3)$, yielding the Dyck path $w = EENENNEENN$ as in Figure~\ref{fig:PF_to_Dyck} below. Here we label the $i$-th $E$ step with the value $p_i = h_i + 1$ for each $i \in [n]$. We can check that $\disp{p} = (0, 1, 1, 0, 1)$ (see also Fact~\ref{fact:inc_pf_inc_outcome}), which gives $\vert \disp{p} \vert = 3 = \area{w}$ (given by the shaded lattice squares).
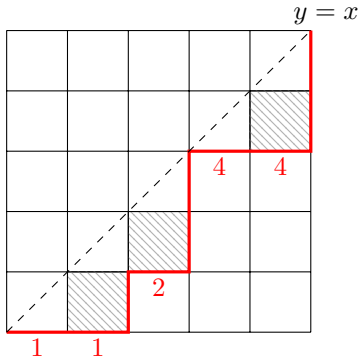
\begin{figure}[ht]
 \centering
  \begin{tikzpicture}[scale=0.4]
  \centering
   \draw [step=2] (0,0) grid (10,10);
    \draw [thin, pattern=north west lines, pattern color=mygray] (2,0)--(4,0)--(4,2)--(2,2)--cycle;
    \draw [thin, pattern=north west lines, pattern color=mygray] (4,2)--(6,2)--(6,4)--(4,4)--cycle;
    \draw [thin, pattern=north west lines, pattern color=mygray] (8,6)--(10,6)--(10,8)--(8,8)--cycle;
    \draw [very thick, color=red]  (0,0)--(4,0)--(4,2)--(6,2)--(6,6)--(10,6)--(10,10);
    \node [red, below, yshift=1pt] at (1,0) {$1$};
    \node [red, below, yshift=1pt] at (3,0) {$1$};
    \node [red, below, yshift=1pt] at (5,2) {$2$};
    \node [red, below, yshift=1pt] at (7,6) {$4$};
    \node [red, below, yshift=1pt] at (9,6) {$4$};
    \draw [dashed] (0,0)--(10,10);
    \node at (10.5,10.5) {$y = x$};
  \end{tikzpicture}
  \caption{The Dyck path corresponding to the non-decreasing parking function $p =  (1, 1, 2, 4, 4) \in \incPF{5}$.\label{fig:PF_to_Dyck}}
\end{figure}
\end{example}

\begin{remark}\label{rmk:area_Naples}
Because of this correspondence to the area statistic of Dyck paths, the total displacement statistic is also sometimes referred to as the \emph{area} of parking functions. This statistic has been studied in previous work, including by Kreweras~\cite{Kre} who showed that it is equi-distributed with the inversion statistic on labelled plane trees. More recently, Colmenarejo \emph{et al.}~\cite{HarrisKNaple} studied the area statistic on a generalisation of parking functions called \emph{$k$-Naples} parking functions where, if a car's preferred spot is occupied, it is first allowed to reverse up to $k$ spots before driving on.
\end{remark}

We end this section with the simple following fact, which will prove useful in Section~\ref{sec:main}.

\begin{fact}\label{fact:inc_pf_inc_outcome}
Let $p \in \incPF{n}$ be a non-decreasing parking function. Then we have $\OO{p} = (1, 2, \ldots, n)$.
\end{fact}

\begin{proof}
First, observe that for any parking function $p = (p_1, \ldots, p_n)$, if two cars $i < j$ have preferences $p_i \leq p_j$, and final parking spots $o_i$ and $o_j$, then we have $o_i < o_j$. Indeed, by definition, $o_i$ is the first spot $k \geq p_i$ which is available when car $i$ enters the car park. In particular, once $i$ has parked, all spots between $p_i$ and $o_i$ (both included) are occupied. Since car $j$ enters after car $i$, and parks in the first available spot $k' \geq p_j \geq p_i$, this spot must be after $o_i$, as desired. Now assume that $p$ is non-decreasing. From the previous observation we see that the sequence $\OO{p}$ must be increasing, and $(1, 2, \ldots, n)$ is the only increasing sequence of $[n]$ of length $n$.
\end{proof}

\subsection{\luk\ paths}\label{subsec:luk}

\begin{definition}\label{def:luk_word}
A \luk\ \emph{word} of length $n$ is a sequence $\ell = (\ell_1, \ldots, \ell_n)$ of integers $\ell_i \geq -1$ such that:
\begin{itemize}[topsep=2pt, noitemsep]
\item For any $k \in [n]$, we have $\sum\limits_{i=1}^k \ell_i \geq 0$.
\item We have $\sum\limits_{i=1}^n \ell_i = 0$.
\end{itemize}
\end{definition}

\luk\ words are usually represented as certain lattice paths, by associating to each $\ell_i$ the step $(1, \ell_i)$. In this work, we choose a slightly different representation by instead taking steps of the form $s^k := (k+1,k)$. We will refer to $k$ as the \emph{size} of the step $s^k$. In this setting, a \luk\ path is a lattice path with $n$ steps in the set $S = \{s^k\}_{k \geq -1}$, which starts at $(0,0)$, ends at $(n,0)$ and never goes below the $x$-axis.  Figure~\ref{fig:exa_extendluk} shows an example of a \luk\ path with $n=12$ steps. There is an obvious bijection between \luk\ words of length $n$ and \luk\ paths with $n$ steps by mapping each element $\ell_i$ in the word to the step $s^{\ell_i}$ in the path. With slight abuse of notation, we identify these two sets, denoting them $\Luk{n}$, i.e.\ we write $\ell \in \Luk{n}$ to refer to a \luk\ word or path, depending on context.

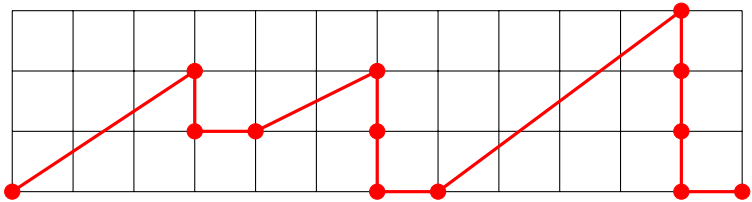
\begin{figure}[ht]
 \centering
  \begin{tikzpicture}[scale=0.4]
  \centering
   \draw [step=2] (0,-10) grid (24,-16);

	\tdot{0}{-16}{red}
    \tdot{6}{-12}{red}
    \tdot{6}{-14}{red}
    \tdot{8}{-14}{red}
    \tdot{12}{-12}{red}
    \tdot{12}{-14}{red}
    \tdot{12}{-16}{red} 
    \tdot{14}{-16}{red}
    \tdot{22}{-10}{red}
    \tdot{22}{-12}{red}
    \tdot{22}{-14}{red}
    \tdot{22}{-16}{red}
    \tdot{24}{-16}{red}
    
    \draw [very thick, color=red]  (0,-16)--(6,-12);
    \draw [very thick, color=red] (6,-12)--(6,-14)--(8,-14);
    \draw [very thick, color=red] (8,-14)--(12,-12);
    \draw [very thick, color=red]  (12,-12)--(12,-14)--(12,-16);
	\draw [very thick, color=red] (12,-16)--(14,-16)--(22,-10);
	\draw [very thick, color=red]  (22,-10)--(22,-12)--(22,-14)--(22,-16);
	\draw [very thick, color=red]  (22,-16)--(24,-16);

  \end{tikzpicture}
  \caption{Example of a \luk\ path corresponding to the word $\ell = (2, -1, 0, 1, -1, -1, 0, 3, -1, -1, -1, 0)$.\label{fig:exa_extendluk}}
\end{figure}

\luk\ paths were named after the Polish mathematician Jan \luk. While maybe not as ubiquitous as their lattice path cousins, Dyck paths and Motzkin paths, they have nonetheless been a rich research topic in combinatorics and discrete probability. Perhaps the most famous use of \luk\ words is as a bijective encoding of rooted plane trees, see e.g.~\cite[Chapter~1, Section~5]{Flaj}. For this, we map a rooted plane tree $T$ with $n+1$ nodes to the word $\ell = (\ell_1, \ldots, \ell_n)$, where $\ell_i$ is one less than the number of children of the $i$-th node visited in the depth-first search (DFS) of $T$ (the last node visited in a DFS is always a leaf, so we omit it in the corresponding word). The bijection implies in particular that $\vert \Luk{n} \vert = \Cat{n}$. Figure~\ref{fig:tree_luk} shows an example of this encoding, with the nodes on the plane tree labelled according to their DFS index. 

\begin{figure}[ht]
 \centering
  \begin{tikzpicture}[scale=0.4]
   \begin{scope}[shift={(-15,0)}]
   \node [draw, circle] (1) at (-8, -17) {$1$};
   \node [draw, circle] (2) at (-11, -14) {$2$};
   \node [draw, circle] (3) at (-8, -14) {$3$};
   \node [draw, circle] (5) at (-4, -14) {$5$};
   \node [draw, circle] (4) at (-8, -11) {$4$};
   \node [draw, circle] (6) at (-5.2, -11) {$6$};
   \node [draw, circle] (7) at (-2.8, -11) {$7$};
   \draw [very thick, blue] (2)--(1)--(3)--(4);
   \draw [very thick, blue] (7)--(5);
   \draw [very thick, blue] (1)--(5)--(6);
   \end{scope}
   
   \begin{scope}[shift={(-1,0)}]
   \node at (-14.5, -14) {\LARGE$\longrightarrow$};
   \node[scale=1.25] at (-8, -14) {$(2, -1, 0, -1, 1, -1)$};
   \node at (-1.5, -14) {\LARGE$\longrightarrow$};
   \end{scope}
   
   \draw [step=2] (0,-12) grid (12,-16);
	\tdot{0}{-16}{red}
    \tdot{6}{-12}{red}
    \tdot{6}{-14}{red}
    \tdot{8}{-14}{red}
    \tdot{8}{-16}{red}
    \tdot{12}{-14}{red}
    \tdot{12}{-16}{red}    
    \draw [very thick, color=red]  (0,-16)--(6,-12)--(6,-14)--(8,-14)--(8,-16)--(12,-14)--(12,-16);
  \end{tikzpicture}
  \caption{A plane tree with nodes labelled according to the DFS (left) together with its Lukasiewicz word encoding (middle) and the corresponding representation as a Lukasiewicz path (right).\label{fig:tree_luk}}
\end{figure}
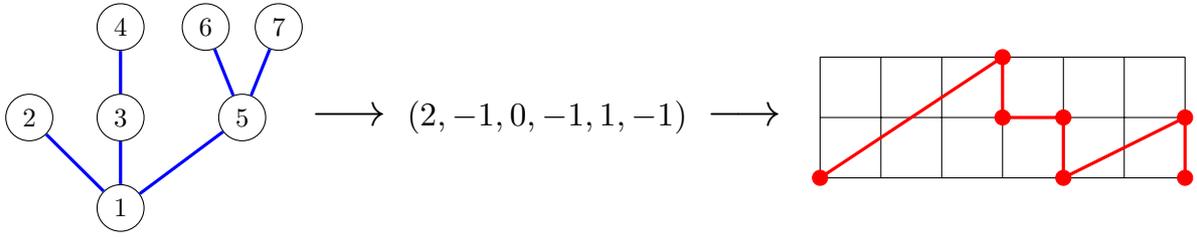

One statistic of interest for a \luk\ path $\ell$ is its \emph{height}, denoted $\height{\ell}$, defined to be the largest $y$-coordinate reached by the path. For example, the \luk\ path in Figure~\ref{fig:exa_extendluk} has height $3$. Another statistic of interest is the \emph{area} of $\ell$, defined to be the area between the path and the $x$-axis. We denote this $\area{\ell}$. This can be computed through the following observation. If $\ell$ takes a step $s^k = (k+1,k)$ starting at height $h$ above the $x$-axis (see Figure~\ref{fig:exa_area}), then the area $\mathcal{A}(k,h)$ ``under'' this step is simply the area of a trapezium, given by 
\beq\label{eq:area_step}
\mathcal{A}(k,h) = (k+1) \cdot \dfrac{h + (h+k)}{2} = \dfrac{(k+1)(2h+k)}{2}.
\eeq
To obtain the area of the entire path, we simply take the sums of all these areas. For example, for  the \luk\ path $\ell$ in Figure~\ref{fig:exa_extendluk}, we get $\area{\ell} = 3 + 0 + 1 + 2 \cdot 3/2 + 0 + 0 + 0 + 4 \cdot 3/2 + 0 + 0 + 0 + 0 = 13$.

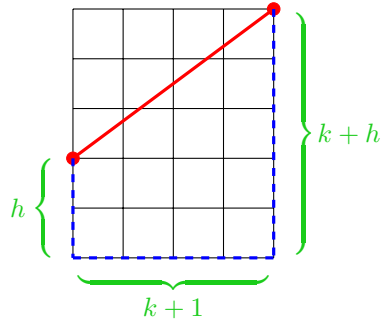
\begin{figure}[ht]
 \centering
  \begin{tikzpicture}[scale=0.33]
  \centering
   \draw [step=2] (0,-6) grid (8,-16);

	\tdot{0}{-12}{red}
    \tdot{8}{-6}{red}
    
    \draw [very thick, color=red]  (0,-12)--(8,-6);
    \draw [very thick, dashed, color=blue] (0,-16)--(8,-16);
    \draw [very thick, dashed, color=blue] (0,-12)--(0,-16); 
	\draw [very thick, dashed, color=blue] (8,-6)--(8,-16);    
    
	\node[rotate = 90, mygreen] at (9.2, -11) {$\underbrace{\hspace{3.2cm}}$};
    \node[mygreen] at (11,-11) {$k+h$};
    \node[rotate = -90, mygreen] at (-1.2, -14) {$\underbrace{\hspace{1.3cm}}$};
    \node[mygreen] at (-2.2,-14) {$h$};
    \node[rotate = 0, mygreen] at (4, -17) {$\underbrace{\hspace{2.5cm}}$};
    \node[mygreen] at (4,-18) {$k+1$};
  \end{tikzpicture}
  \caption{A step $s^k = (k+1, k)$ starting at height $h$. The area under this step is given by Equation~\eqref{eq:area_step}.\label{fig:exa_area}}
\end{figure}

\section{The main result}\label{sec:main}

We are now equipped to define the map from parking functions to \luk\ paths. Given a parking function $p = (p_1, \ldots, p_n) \in \PF{n}$, we define a sequence $\PFtoLuk(p) = \ell = (\ell_1, \ldots, \ell_n)$ by:
\beq\label{eq:PF_extendluk_path}
\forall i \in [n], \ \ell_i := \left\vert \{j \in [n]; \, p_j = i \} \right\vert - 1.
\eeq
In other words, $\ell_i$ is one less than the number of cars whose preferred spot is $i$ in the parking function $p$. In particular, we have $\ell_i \geq -1$. The following is then a straightforward consequence of Proposition~\ref{pro:pf_char} (Characterisation~(3)) and of Definition~\ref{def:luk_word}.

\begin{proposition}\label{pro:pf_to_luk_well-def}
For any parking function $p \in \PF{n}$, we have $\PFtoLuk(p) \in \Luk{n}$.
\end{proposition}

We are now equipped to state the main result of this paper.

\begin{theorem}\label{thm:bij_pf_luk}
The map $\PFtoLuk: \incPF{n} \rightarrow \Luk{n}$ is a bijection. Moreover, for any $p \in \incPF{n}$, we have $\vert \disp{p} \vert = \area{\PFtoLuk(p)}$, and $\max \left( \disp{p} \right) = \height{\PFtoLuk(p)}$.
\end{theorem}

\begin{example}\label{exa:p_to_extendluk}
Consider the non-decreasing parking function $p = (1, 1, 1, 3, 4, 4, 7, 8, 8, 8, 8, 12) \in \incPF{12}$. For each car $i$, the displacement $d_i$ is given by $d_i = o_i - p_i = i - p_i$ by Fact~\ref{fact:inc_pf_inc_outcome}, yielding $\disp{p} = (0, 1, 2, 1, 1, 2, 0, 0, 1, 2, 3, 0)$. In particular, we get total displacement $\vert \disp{p} \vert = 13$, and maximum displacement $\max \left( \disp{p} \right) = 3$.  The corresponding \luk\ word is given by $\ell = \PFtoLuk(p) =  (2,-1,0,1,-1,-1,0,3,-1,-1,-1,0)$, whose lattice path is exactly the \luk\ path illustrated in Figure~\ref{fig:exa_extendluk}, which has $\area{\ell} = 13$ and $\height{\ell} = 3$, as desired.
\end{example}

\begin{remark}\label{rmk:luk_dyck_bijs}
The bijection from non-decreasing parking functions to Dyck paths from Section~\ref{subsec:pf} can be similarly defined if we consider a Dyck path to be a lattice path from $(0,0)$ to $(2n,0)$ with steps $U = (1,1)$ and $D = (1,-1)$ which never goes below the $x$-axis. In this setting, the Dyck path $w$ corresponding to $p \in \incPF{n}$ is defined by $w = U^{q_1} D U^{q_2} D \cdots U^{q_n} D$, where $q_i := \left\vert \{j \in [n]; \, p_j = i \} \right\vert$ for each $i \in [n]$. Compared to this construction, our map has the disadvantage of needing to subtract one from each $q_i$ (see Equation~\eqref{eq:PF_extendluk_path}). However, the benefits of our map are that we directly get the sequence of steps in the corresponding path (without needing to insert the $D$ steps as in the Dyck path), as well as more direct mappings from the total and maximum displacement statistics to the area and height of the path.
\end{remark}

To show that $\PFtoLuk$ is a bijection, we exhibit its inverse in Algorithm~\ref{algo:LuktoPF}. We will use square brackets instead of standard parentheses to delimit sequences, using the latter instead to indicate precedence order (as on Line~\ref{line:new_p}). Here $\varepsilon := [\,]$ denotes the empty sequence, and for two sequences $a = [a_1, \ldots, a_x]$, $b = [b_1, \ldots, b_y]$, and an integer $m$, we write $a + b := [a_1, \ldots, a_x, b_1, \ldots, b_y]$ for the concatenation of $a$ and $b$, and $a * m := a + \cdots + a$ (repeated $m$ times), with the convention $a * 0 = \varepsilon$. As usual, multiplication takes precedence over addition.

\begin{algorithm}
\caption{From \luk\ words to parking functions}\label{algo:LuktoPF}
  \begin{algorithmic}[1]
      \Require $\ell = [\ell_1, \ldots, \ell_n] \in \Luk{n}$
      \State \textbf{Initialise:} $p \gets \varepsilon$
      \For {$i = 1$ to $n$} 
        \State $p \gets p + [i] * (\ell_i + 1)$ \label{line:new_p} \Comment{Append $i$ to $p$ $(\ell_i + 1)$ times}
      \EndFor \\
      \Return $p := \LuktoPF(\ell)$
  \end{algorithmic}
\end{algorithm}

\begin{theorem}\label{thm:algo_LuktoPF}
For any \luk\ word $\ell \in \Luk{n}$, Algorithm~\ref{algo:LuktoPF} outputs a non-decreasing parking function $p = \LuktoPF(\ell) \in \incPF{n}$. Moreover, the maps $\PFtoLuk: \incPF{n} \rightarrow \Luk{n}$ and $\LuktoPF: \Luk{n} \rightarrow \incPF{n}$ are inverses of each other.
\end{theorem}

\begin{remark}\label{rem:luk_to_p}
The non-decreasing parking function $p = (p_1, \ldots, p_n) = \LuktoPF(\ell)$ corresponding to a \luk\ path $\ell \in \Luk{n}$ can also be read ``graphically'' as follows. For each $i \in [n]$, $p_i$ is the index $j$ of the step $\ell_j$ of $\ell$ which crosses the region from $x=i-1$ to $x=i$. Figure~\ref{fig:exa_luk_to_pf} illustrates this construction for the \luk\ path $\ell = [2,-1,0,1,-1,-1,0,3,-1,-1,-1,0]$. The parking function $p = \LuktoPF(\ell)$ is obtained by reading the blue indices below the $x$-axis from left-to-right, yielding $p = (1,1,1,3,4,4,7,8,8,8,8,12)$. It is straightforward to check that Algorithm~\ref{algo:LuktoPF} gives the same parking function.

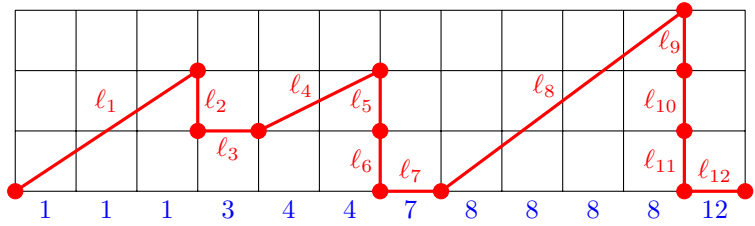
\begin{figure}[ht]
 \centering
  \begin{tikzpicture}[scale=0.4]
  \centering
   \draw [step=2] (0,-10) grid (24,-16);

	\tdot{0}{-16}{red}
    \tdot{6}{-12}{red}
    \tdot{6}{-14}{red}
    \tdot{8}{-14}{red}
    \tdot{12}{-12}{red}
    \tdot{12}{-14}{red}
    \tdot{12}{-16}{red} 
    \tdot{14}{-16}{red}
    \tdot{22}{-10}{red}
    \tdot{22}{-12}{red}
    \tdot{22}{-14}{red}
    \tdot{22}{-16}{red}
    \tdot{24}{-16}{red}
    
    \draw [very thick, color=red]  (0,-16)--(6,-12);
    \draw [very thick, color=red] (6,-12)--(6,-14)--(8,-14);
    \draw [very thick, color=red] (8,-14)--(12,-12);
    \draw [very thick, color=red]  (12,-12)--(12,-14)--(12,-16);
	\draw [very thick, color=red] (12,-16)--(14,-16)--(22,-10);
	\draw [very thick, color=red]  (22,-10)--(22,-12)--(22,-14)--(22,-16);
	\draw [very thick, color=red]  (22,-16)--(24,-16);
	
	\node [red] at (3,-13) {$\ell_1$};
	\node [right, xshift=-1pt, red] at (6,-13) {$\ell_2$};
	\node [below, yshift=1pt, red] at (7,-14) {$\ell_3$};
	\node [red] at (9.4,-12.5) {$\ell_4$};
	\node [left, xshift=1pt, red] at (12,-13) {$\ell_5$};
	\node [left, xshift=1pt, red] at (12,-15) {$\ell_6$};
	\node [above, red] at (13,-16) {$\ell_7$};
	\node [red] at (17.4,-12.5) {$\ell_8$};
	\node [left, xshift=2.5pt, yshift=-0.5pt, red] at (22,-11) {$\ell_9$};
	\node [left, xshift=1pt, red] at (22,-13) {$\ell_{10}$};
	\node [left, xshift=1pt, red] at (22,-15) {$\ell_{11}$};
	\node [above, red] at (23,-16) {$\ell_{12}$};
	
	\node [below, blue] at (1,-16) {$1$};
	\node [below, blue] at (3,-16) {$1$};
	\node [below, blue] at (5,-16) {$1$};
	\node [below, blue] at (7,-16) {$3$};
	\node [below, blue] at (9,-16) {$4$};
	\node [below, blue] at (11,-16) {$4$};
	\node [below, blue] at (13,-16) {$7$};
	\node [below, blue] at (15,-16) {$8$};
	\node [below, blue] at (17,-16) {$8$};
	\node [below, blue] at (19,-16) {$8$};
	\node [below, blue] at (21,-16) {$8$};
	\node [below, blue] at (23,-16) {$12$};

  \end{tikzpicture}
  \caption{Illustrating the construction from \luk\ paths to non-decreasing parking functions.\label{fig:exa_luk_to_pf}}
\end{figure}

\end{remark}

\begin{proof}[Proof of Theorem~\ref{thm:algo_LuktoPF}]
By construction, Algorithm~\ref{algo:LuktoPF} outputs a non-decreasing sequence $p$ of elements in $[n]$, whose length is $\sum\limits_{i=1}^n (\ell_i + 1) = 0 + n = n$, where we use the fact that the elements of a \luk\ word sum to $0$. Therefore $p$ is a parking preference. 
Now fix some $i \in [n]$. By construction, we have $\left\vert \{ j \in [n]; \, p_j \leq i \} \right\vert = \sum\limits_{j = 1}^i (\ell_j + 1) \geq 0 + i = i$ (again applying Definition~\ref{def:luk_word}). 
Therefore $p$ is a parking function by Proposition~\ref{pro:pf_char}. The fact that the maps $\PFtoLuk: \incPF{n} \rightarrow \Luk{n}$ and $\LuktoPF: \Luk{n} \rightarrow \incPF{n}$ are inverses of each other follows immediately from their constructions.
\end{proof}

We now turn to the proof of Theorem~\ref{thm:bij_pf_luk}. Theorem~\ref{thm:algo_LuktoPF} implies that the map $\PFtoLuk: \incPF{n} \rightarrow \Luk{n}$ is a bijection, so it remains to show the equalities relating to the displacement statistic. For a \luk\ path $\ell = (\ell_1, \ldots, \ell_n) \in \Luk{n}$ and an index $j \in [n]$, we denote by $h(\ell; j) := \sum\limits_{i=1}^j \ell_i$ the height of the path $\ell$ after $j$ steps.

\begin{lemma}\label{lem:height_excess_cars}
Let $p \in \incPF{n}$ be a non-decreasing parking function, and $\ell := \PFtoLuk(p)$ the corresponding \luk\ path. For any $j \in \{0, \ldots, n\}$, we have:
\beq\label{eq:height_excess_cars}
h(\ell; j) = \left\vert \{i \in [n]; \, p_i \leq j \} \right\vert - j.
\eeq
\end{lemma}

\begin{proof}
By definition, for any $i \in [n]$ we have $\ell_i = \left\vert \{k \in [n]; \, p_k = i \} \right\vert -1$. Equation~\eqref{eq:height_excess_cars} then follows through summation.
\end{proof}

One may think of the right-hand side of Equation~\eqref{eq:height_excess_cars} as measuring \emph{excess cars}: it counts the number of cars that wish to park in or before spot $j$ but will be unable to do so (i.e.\ end up parking in some spot $k > j$). Lemma~\ref{lem:height_excess_cars} then states that $\PFtoLuk$ maps the excess cars statistic to the height of the corresponding \luk\ path.

\begin{proof}[Proof of Theorem~\ref{thm:bij_pf_luk}]
Let $p \in \incPF{n}$ be a non-decreasing parking function, and $\ell := \PFtoLuk(p)$ the corresponding \luk\ path. Fix some $j \in \{0, \ldots, n-1\}$, and define 
$\gamma_j := \left\vert \{i \in [n]; \, p_i \leq j \} \right\vert $ to be the number of cars which prefer to park in one of the first $j$ spots. 
Since $p$ is non-decreasing, Lemma~\ref{lem:height_excess_cars} and Fact~\ref{fact:inc_pf_inc_outcome} then imply that cars $1$ to $\gamma_j$ occupy exactly the spots $1$ to $j + h$, where $h := h(\ell; j)$. 
Now let $k := \ell_{j+1}$, and consider the cars $\gamma_j + 1, \cdots, \gamma_j+k+1$. By construction, this is exactly the set of cars whose preferred spot is $j+1$. In the parking process for $p$, these cars occupy spots $j+h+1, \cdots, j+h+k+1$. This yields the (partial) displacement vector
\beq\label{eq:disp_partial}
\left( d_{\gamma_j+1}, \ldots, d_{\gamma_j+k+1} \right) = (h, \ldots, h+k).
\eeq
By summation, we get:
\begin{align*}
d_{\gamma_j+1} + \cdots + d_{\gamma_j+k+1} & = h + \cdots + (h+k) = (k+1)h + 1 + \cdots + k \\
 & = (k+1)h + \dfrac{k(k+1)}{2} = \dfrac{(k+1)(2h+k)}{2}, \\
\end{align*}
which is exactly the area $\mathcal{A}(k,h)$ given in Equation~\eqref{eq:area_step}. In words, the total displacement of cars preferring spot $(j+1)$ is equal to the area under the $(j+1)$-th step of the \luk\ path $\ell$, which immediately gives $\vert \disp{p} \vert = \area{\ell}$ by summing over all steps. Moreover, in Equation~\eqref{eq:disp_partial}, we may also take the maximum to get
$$ \max \left( d_{\gamma_j+1}, \ldots, d_{\gamma_j+k+1} \right) = h + k = h(\ell; j) + \ell_{j+1} = h(\ell; j+1). $$
In other words, for any $j \in \{0, \ldots, n-1\}$, the height $h(\ell; j+1)$ of the \luk\ path $\ell$ after $j+1$ steps is equal to the maximum displacement of cars $\gamma_j + 1, \cdots, \gamma_j+k+1$, which as noted above are exactly those that prefer spot $j+1$. Taking the maximum over all such $j$ immediately yields $\max \left( \disp{p} \right) = \height{\PFtoLuk(p)}$, as desired.
\end{proof}

\begin{remark}\label{rem:label_luk}
Theorem~\ref{thm:bij_pf_luk} essentially states that non-decreasing parking functions are uniquely defined by the numbers of cars preferring each spot in the car park. In other words, \luk\ paths encode parking functions up to permutation of their elements. In order to encode all parking functions, we need to also know which cars prefer a given spot. This can be done by labelling each step $s^k = (k+1, k)$ for $k \geq 0$ of the \luk\ path with a subset $S \subset [n]$ of size $k+1$ such that the label sets over all steps form a partition of the set $[n]$. For example, consider the \luk\ path $\ell$ from Figure~\ref{fig:exa_extendluk}. We may choose the following labelling: $\left(2^{\{2, 3, 10\}}, -1, 0^{\{5\}}, 1^{\{6,8\}}, -1, -1, 0^{\{9\}}, 3^{\{1,7,11,12\}}, -1, -1, -1, 0^{\{4\}} \right)$, with the exponent indicating the label set associated to each step. This encodes the parking function $p = (8, 1, 1, 12, 3, 4, 8, 4, 7, 1, 8, 8)$. 
\end{remark}

In the above example, we have $\OO{p} = (8, 1, 2, 12, 3, 4, 9, 5, 7, 6, 10, 11)$, yielding the displacement vector $\disp{p} = (0, 0, 1, 0, 0, 0, 1, 1, 0, 5, 2, 3)$, so the maximum displacement is $5$, while the path has height $3$. This means that in general, the map $\PFtoLuk$ does not map the maximum displacement of a parking function to the height of its corresponding path. On the other hand, the total displacement is still equal to $13$. We will see that this property holds true in general.

\begin{theorem}\label{thm:surj_PF_luk}
The map $\PFtoLuk : \PF{n} \rightarrow \Luk{n}$ is a surjection. Moreover, for any $p \in \PF{n}$, we have $\vert \disp{p} \vert = \area{\PFtoLuk(p)}$. Finally, for any \luk\ path $\ell$, the fibre set $\PFtoLuk^{-1}(\ell) := \{ p \in \PF{n}; \, \PFtoLuk(p) = \ell \}$ is obtained by taking all possible permutations of the non-decreasing parking function $\LuktoPF(\ell)$.
\end{theorem}

\begin{proof}
The surjectivity of $\PFtoLuk$ and description of its fibres follow from Theorem~\ref{thm:bij_pf_luk} and Remark~\ref{rem:label_luk}. We show that for any $p \in \PF{n}$, we have $\vert \disp{p} \vert = \area{\PFtoLuk(p)}$. In fact, since this formula holds for non-decreasing parking functions by Theorem~\ref{thm:bij_pf_luk}, it suffices to show that the total displacement $\vert \disp{p} \vert$ is invariant under permutation of parking preferences. 
But this follows essentially from the definition of the displacement, combined with the observation that if $o = \OO{p} = (o_1, \ldots, o_n)$ is the outcome of a parking function $p$, every spot in $[n]$ appears exactly once in $o$. Then we get: 
$$\vert \disp{p} \vert = \sum\limits_{i=1}^n (o_i - p_i) = \sum\limits_{i=1}^n o_i - \sum\limits_{i=1}^n p_i = \frac{n(n+1)}{2} - \sum\limits_{i=1}^n p_i, $$
and the right-hand side is clearly invariant under permutation.
\end{proof}

\section{Specialisations}\label{sec:spec}

There are a number of natural restrictions that we can place on parking functions. For example, if we restrict each car to have displacement at most one, we get so-called \emph{unit-interval} parking functions (see Section~\ref{subsec:UPF} for a short discussion on what is known about these). Conversely, \luk\ paths can also be restricted, for example in terms of their height, or the largest step size allowed. In this section, we study several such restrictions under the bijections $\PFtoLuk$ and $\LuktoPF$.

\subsection{The Motzkin family}\label{subsec:Motz}

If we restrict steps in a \luk\ path to have size at most $1$, we get the well-known \emph{Motzkin paths} (see e.g.~\cite{Motz}). We denote by $\Motz{n}$ the set of Motzkin paths with $n$ steps. These are counted by the Motzkin numbers (Sequence~A001006 in the OEIS~\cite{OEIS}). The corresponding parking functions $p = (p_1, \ldots, p_n) \in [n]^n$ are those that satisfy the restriction
\beq\label{eq:MotzPF}
\forall i \in [n], \ \vert \{ j \in [n];\, p_j = i \} \vert \leq 2.
\eeq
In other words, every spot in the car park is preferred by at most $2$ cars. These were studied in previous work by the authors~\cite[Section~3]{SZ_MVP} under the name of \emph{Motzkin parking functions}. In particular, they provided a bijection between non-crossing matchings and parking functions whose \emph{MVP outcome} reverses the order of the cars. Here, the MVP outcome of a parking function is the order in which cars end up parking if they follow the MVP (Most Valuable Player) parking process defined by Harris \emph{et al.}~\cite{HarrisMVP}. In this process, when a car finds its preferred
spot occupied by a previous car, it “bumps” that car out of the spot and parks there. The earlier car then has to drive on, and parks in the first available spot it can find.

\subsection{Prime parking functions}\label{subsec:primePF}

Given a parking function $p = (p_1, \ldots, p_n) \in \PF{n}$, and an index $j \in [n]$, we say that $j$ is a \emph{breakpoint} for $p$ if $\vert \{ i \in [n]; \, p_i \leq j \} \vert = j$, i.e.\ exactly $j$ cars prefer the first $j$ spots. A parking function is said to be \emph{prime} if its only breakpoint is at index $n$. The concept of prime parking functions was introduced by Gessel, who showed that there are $(n-1)^{(n-1)}$ prime parking functions (see e.g.~\cite[Exercise~5.49]{StanEC}). A bijective proof of this formula was later given in~\cite{DO_primePF}. We denote by $\PPF{n}$, respectively $\incPPF{n}$, the set of prime parking functions, respectively non-decreasing prime parking functions, of length $n$. In general, breakpoints of parking functions are easily read from the corresponding \luk\ path.

\begin{proposition}\label{pro:breakpoint_PF_luk}
Let $p \in \PF{n}$ be a parking function and $j \in [n]$ an index. Then $j$ is a breakpoint for $p$ if and only if the \luk\ path $\ell := \PFtoLuk(p)$ hits the $x$-axis after $j$ steps, i.e.\ $h(\ell; j) = 0$.
\end{proposition}

\begin{proof}
By construction, if $\ell = (\ell_1, \ldots, \ell_n)$, then the height of the path after $j$ steps is simply 
$$h(\ell; j) = \sum\limits_{i=1}^j \ell_i = \sum\limits_{i=1}^j \left( \vert \{ k \in [n]; \, p_k = i \} \vert - 1 \right) = \vert \{ i \in [n]; \, p_i \leq j \} \vert - j,$$
where we applied the definition of $\PFtoLuk$ from Equation~\eqref{eq:PF_extendluk_path}. The result immediately follows.
\end{proof}

We say that a \luk\ path $\ell \in \Luk{n}$ is \emph{prime} if it stays strictly above the $x$-axis other than at its start and end points $(0,0)$ and $(n,0)$, and denote by $\PLuk{n}$ the set of prime \luk\ paths. We now state our first specialisation.

\begin{theorem}\label{thm:bij_primePF_luk}
The map $\PFtoLuk : \incPPF{n} \rightarrow \PLuk{n} $ is a bijection. Moreover, we have $\vert \incPPF{n} \vert = \vert \PLuk{n} \vert = \Cat{n-1}$.
\end{theorem}

\begin{proof}
That $\PFtoLuk$ induces a bijection from non-decreasing prime parking functions to prime \luk\ paths is an immediate consequence of Proposition~\ref{pro:breakpoint_PF_luk}. To get the enumeration, notice that a non-decreasing parking sequence $p = (p_1, \ldots, p_n)$ is a prime parking function if and only if we have $p_1= 1$ and $p_i < i$ for all $i \geq 2$ (applying Proposition~\ref{pro:pf_char}, Case~(2), and the definition of prime parking functions). This immediately implies that the map $(p_1, \ldots, p_n) \mapsto (p_2, \ldots, p_n)$ is a bijection from $\incPPF{n}$ to $\incPF{n-1}$, yielding the desired enumeration.
\end{proof}

\begin{remark}\label{rem:primePF_primeLuk_count}
We can also express the above bijection $\incPPF{n} \rightarrow \incPF{n-1}$ in terms of \luk\ paths. We get the bijection $\PLuk{n} \rightarrow \Luk{n-1}$, $(\ell_1, \ldots, \ell_n) \mapsto (\ell_1 - 1, \ell_2, \ldots, \ell_{n-1})$, from prime \luk\ paths of length $n$ to \luk\ paths of length $n-1$. In words, given a prime \luk\ path, we decrease the size of its first step by one, and delete the last step (this is necessarily a ``down'' step, since the path is prime), yielding a \luk\ path with one less step. This bijection is illustrated in Figure~\ref{fig:bij_luk_primeLuk} for the prime \luk\ path $\ell = (3, -1, -1, 1, 0, -1, 0, -1)$ (left), which maps to the \luk\ path $\ell' = (2, -1, -1, 1, 0, -1, 0)$ (right).
\end{remark}

\begin{figure}[ht]
 \centering
  \begin{tikzpicture}[scale=0.4]
   \draw [step=2] (0,0) grid (16,6);
	\tdot{0}{0}{red}
	\tdot{8}{6}{red}
	\tdot{8}{4}{red}
	\tdot{8}{2}{red}
	\tdot{12}{4}{red}
	\tdot{14}{4}{red}
	\tdot{14}{2}{red}
	\tdot{16}{2}{red}
	\tdot{16}{0}{red}    
    \draw [very thick, color=red]  (0,0)--(8,6)--(8,4)--(8,2)--(12,4)--(14,4)--(14,2)--(16,2)--(16,0);
    \node at (19,2) {{\LARGE$\longmapsto$}};
    \begin{scope}[shift={(20,-2)}]
   \draw [step=2] (2,2) grid (16,6);
	\tdot{2}{2}{red}
	\tdot{8}{6}{red}
	\tdot{8}{4}{red}
	\tdot{8}{2}{red}
	\tdot{12}{4}{red}
	\tdot{14}{4}{red}
	\tdot{14}{2}{red}
	\tdot{16}{2}{red}
    \draw [very thick, color=red]  (2,2)--(8,6)--(8,4)--(8,2)--(12,4)--(14,4)--(14,2)--(16,2);
    \end{scope}
  \end{tikzpicture}
  \caption{Illustrating the bijection from $\PLuk{n}$ to $\Luk{n-1}$. \label{fig:bij_luk_primeLuk}}
\end{figure}
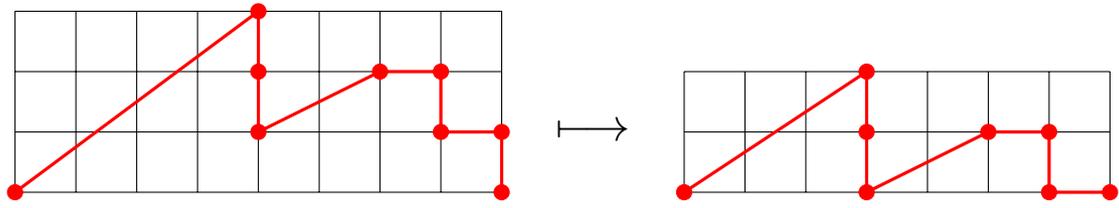

\subsection{Unit-interval parking functions}\label{subsec:UPF}

\begin{definition}\label{def:UPF}
A parking function $p$ is said to be \emph{unit-interval} if it satisfies $\max \left( \disp{p} \right) \leq 1$.
\end{definition}

In other words, a unit-interval parking function is one where each car has displacement zero or one. 
Unit-interval parking functions were originally defined by Hadaway~\cite{Hadaway}, who showed that they are enumerated by the Fubini numbers. The theory of unit-interval parking functions was further developed in~\cite{HarrisUPF1, HarrisUPF2, HarrisUPF3, HarrisUPF4}, with rich combinatorial connections to the permutohedron and a generalisation of the aforementioned Fubini numbers. 
We denote by $\UPF{n}$, respectively $\incUPF{n}$, the set of unit-interval parking functions, respectively unit-interval non-decreasing parking functions, of length $n$. We first state a straightforward characterisation of unit-interval parking functions.

\begin{proposition}\label{pro:UPF_charac}
Let $p = (p_1, \ldots, p_n) \in \PF{n}$ be a parking function with outcome $\OO{p} = (o_1, \ldots, o_n)$. Then $p \in \UPF{n}$ if and only if we have $o_i \in \{p_i, p_i + 1 \}$ for all $i \in [n]$. That is, each car parks either in its preferred spot, or in the spot immediately after.
\end{proposition}

Now consider a unit-interval parking function $p \in \UPF{n}$, and its corresponding \luk\ path $\ell := \PFtoLuk(p)$. By Theorem~\ref{thm:bij_pf_luk} and Definition~\ref{def:UPF}, we have $\height{\ell} = \max \left( \disp{p} \right) \leq 1$. In particular, all steps in $\ell$ have size at most one, making $\ell$ a Motzkin path, or equivalently $p$ is a Motzkin parking function in the sense of Section~\ref{subsec:Motz}. This can also be seen directly from the parking function $p$. Indeed, in any parking function, if three cars prefer the same spot, then the last of these cars to arrive must have displacement at least two. We then get the following specialisation. To simplify notation, we write $\Motz{n}^{ \leq 1}$ for the set of Motzkin paths of length $n$ and height at most one, and refer to these as $1$-Motzkin paths.

\begin{theorem}\label{thm:bij_UPF_Motz}
The map $\PFtoLuk: \incUPF{n} \rightarrow \Motz{n}^{ \leq 1 }$ is a bijection. Moreover, we have $\vert \incUPF{n} \vert = \vert  \Motz{n}^{ \leq 1 } \vert = 2^{n-1}$.
\end{theorem}

\begin{proof}
That the map is a bijection follows from the preceding remarks and Theorem~\ref{thm:bij_pf_luk}. For the enumeration, note that if $p \in \incPF{n}$, then $p$ is unit-interval if and only if we have $p_1 = 1$, and $p_i \in \{i-1, i\}$ for all $i \geq 2$ (applying Proposition~\ref{pro:UPF_charac}). There are therefore two choices for each car $2$ to $n$, yielding $2^{n-1}$ choices in total.
\end{proof}

The enumeration can also be seen on the Motzkin paths. Indeed, a $1$-Motzkin path $m$ is uniquely characterised by the subset $\{ j \in [n-1]; \, h(m; j) = 1\}$, giving a bijection between $\Motz{n}^{\leq 1}$ and subsets of $[n-1]$.

\begin{remark}\label{rem:UPF_perm_inv}
Unlike prime parking functions, the set of unit-interval parking functions is not \emph{permutation invariant}. That is, there exists a unit-interval parking function $p$ such that permuting the preferences of $p$ no longer yields a unit-interval parking function. In the notation of this paper, this means that there exists a parking function $p \in \PF{n} \setminus \UPF{n}$ such that $\PFtoLuk(p)$ is a $1$-Motzkin path. For example, if we take $p = (1, 2, 1) \in \PF{3}$, then car $3$ has displacement $2$, so $p$ is not unit-interval, but $\PFtoLuk(p)$ is the \luk\ path with steps $(1, 0, -1)$, which is $1$-Motzkin. In particular, this means that there are labellings of a $1$-Motzkin path, in the sense of Remark~\ref{rem:label_luk}, which do not yield unit-interval parking functions.
\end{remark}

\section*{Acknowledgment}

The research leading to these results is partially funded by the National Natural Science Foundation of China (NSFC), grant number 12101505, by the Research Development Fund of Xi'an Jiaotong-Liverpool University, grant number RDF-22-01-089, and by the Postgraduate Research Scholarship of Xi'an Jiaotong-Liverpool University, grant number PGRS2012026.


\end{document}